\documentclass{article}

\usepackage[preprint, nonatbib]{neurips_2023}

\usepackage[utf8]{inputenc} % allow utf-8 input
\usepackage[T1]{fontenc}    % use 8-bit T1 fonts
\usepackage[colorlinks, linkcolor=red, citecolor=blue, urlcolor=blue]{hyperref}       % hyperlinks
\usepackage{url}            % simple URL typesetting
\usepackage{booktabs}       % professional-quality tables
\usepackage{amsfonts}       % blackboard math symbols
\usepackage{nicefrac}       % compact symbols for 1/2, etc.
\usepackage{microtype}      % microtypography
\usepackage{xcolor}         % colors

\usepackage{amsthm}
\theoremstyle{plain}
\newtheorem{theorem}{Theorem}[section]
\newtheorem{lemma}[theorem]{Lemma}
\newtheorem{corollary}[theorem]{Corollary}
\theoremstyle{definition}
\newtheorem{definition}[theorem]{Definition}
\newtheorem*{remark}{Remark}
\usepackage{algpseudocode}
\usepackage{algorithm}

\usepackage[english]{babel}

\usepackage[nottoc]{tocbibind}

\usepackage[square,numbers]{natbib}
\usepackage{graphicx}

\title{Accelerating Convergence in Global Non-Convex Optimization with Reversible Diffusion}

\author{%
  Ryo Fujino\\
  Department of Mathematical Informatics,\\
  Graduate School of Information Science and Technology,\\
  The University of Tokyo, Tokyo, Japan\\
  \texttt{ryo-fujino@g.ecc.u-tokyo.ac.jp} \\
}

\begin{document}

\maketitle

\begin{abstract}
Langevin Dynamics has been extensively employed in global non-convex optimization due to the concentration of its stationary distribution around the global minimum of the potential function at low temperatures. In this paper, we propose to utilize a more comprehensive class of stochastic processes, known as reversible diffusion, and apply the Euler-Maruyama discretization for global non-convex optimization. We design the diffusion coefficient to be larger when distant from the optimum and smaller when near, thus enabling accelerated convergence while regulating discretization error, a strategy inspired by landscape modifications. Our proposed method can also be seen as a time change of Langevin Dynamics, and we prove convergence with respect to KL divergence, investigating the trade-off between convergence speed and discretization error. The efficacy of our proposed method is demonstrated through numerical experiments.
\end{abstract}

\section{Introduction}

The global optimization of non-convex functions is a crucial problem frequently encountered in various domains, particularly machine learning and statistics. However, addressing it in a general context is a daunting task, as standard gradient descent techniques often only yield local optima.

This paper focuses on the unconstrained optimization problem:
\[
\displaystyle \min_{x\in\mathbb{R}^n}F(x)
\]
where we assume that $F: \mathbb{R}^n\to \mathbb{R}$ is twice continuously differentiable, and it is potentially non-convex. Let $x_*$ be the optimal point (solution) of this problem.
Langevin Dynamics has been widely used to tackle this problem:
\[
dX_t = -\nabla F(X_t)dt + \sqrt{2\beta^{-1}}dW_t,\quad \beta>0
\]
where $(W_t)_{t\ge 0}$ denotes Brownian motion and $\beta > 0$ is the inverse temperature parameter. This can be viewed as gradient descent perturbed with Gaussian noise. If we denote the stationary distribution of this stochastic process as $\nu$, it becomes $\nu(x) \propto \exp(-\beta F(x))$ which is called the Gibbs distribution. As $\beta$ approaches infinity (i.e., low temperature limit), this distribution converges towards one that is concentrated around $x_*$, even if $F$ is a non-convex function\cite{Geman,Hwang}. Its convergence in probability through simulated annealing was proved by \cite{Chiang}, and \cite{Kushner} analysed its asymptotic behavior. This property is particularly useful for the global non-convex optimization.

For the purpose of devising optimization algorithms, the Euler-Maruyama discretization of the diffusion has been employed, and its convergence analysis has been extensively researched by \cite{Taiji,Raginsky,Vempala, Xu}. Recently, a more broad class of diffusion process with state-dependent volatility has been suggested for the same optimization objective. It is given by
\[
dX_t = b(X_t)dt + \sqrt{2\beta^{-1}}\sigma(X_t)dW_t,\quad \beta>0 
\]
where $b:\mathbb{R}^n\to\mathbb{R}^n, \sigma:\mathbb{R}^n\to \mathbb{R}^{n\times n}$, and $b,\sigma$ are selected in a way that the diffusion possesses $\nu$ as a stationary distribution\cite{Ma}. 
The convergence analysis and the specific construction of diffusion provided by \cite{Erdogdu} inspires our pursuit to find the appropriate selection of its coefficients, a topic we will explore in the rest of this paper.

\subsection{Our Contributions}
The principal contributions of this paper can be distilled as follows.
We introduce a diffusion process with a non-constant diffusion coefficient for global non-convex optimization. The proposed diffusion is represented by:
\begin{eqnarray}
f_{\lambda, \vartheta}(x) &=& \lambda(1 - e^{-\vartheta x^2}),\nonumber\\
a(x) &=& \sigma(x)^{\top}\sigma(x)
= (f_{\lambda, \vartheta}((F(x) - c)^+)) + 1)I_n,\nonumber\\
b(x) &=& -a(x)\nabla F(x) + \beta^{-1}\nabla\cdot a(x) \nonumber
\end{eqnarray}
where $c > F(x_*), \lambda > 0, \vartheta > 0$ and $I_n$ is an identity matrix. 
Here, we denote the divergence of matrix valued function $a(x)$ as
\[
\nabla \cdot a(x) = \sum_{i=1}^{n}e_i\sum_{j=1}^{n}\frac{\partial a_{ij}(x)}{\partial x_j}
\]
where $e_1,\dots,e_n$ are the standard basis of $\mathbb{R}^n$, and $(\,\cdot\,)^+ = \max\{\,\cdot\,, 0\}$.
In this setting, the diffusion becomes reversible, a term which we will define subsequently. We prove the convergence of the discretized algorithm of this diffusion with respect to the KL divergence, and this result is a generalization of Theorem 1 in \cite{Vempala}. Based on this result, we offer a comprehensive discussion on the trade-off between convergence speed and discretization error.

The construction of this diffusion assumes that we have a prior knowledge of the upper bound for $F(x_*)$  
However, we can adaptively estimate $F(x_*)$ with a current minimum of each iteration and put the estimate into $c$ for practical applications.
We introduce a novel adaptive optimization algorithm (AdaVol) which doesn't need prior knowledge of $F(x_*)$ from the aforementioned diffusion. Through numerical experiments, we illustrate their accelerated convergence in comparison to Langevin Dynamics.

\subsection{Related Works}
Accelerating the convergence of diffusion process towards its statinary distribution is a challenging task that has been tackled by numerous prior works.
Theorem 2.6 in \cite{Wang} and Proposition 3.5 in \cite{Erdogdu} established a criterion for identifying an appropriate diffusion coefficient from the perspective of distant dissipativity, and \cite{Erdogdu} highlighted the effectiveness of non-constant diffusion coefficient through a numerical experiment for a heavy-tailed objective function. 

There are also several prior works that propose concrete methods for constructing a non-constant diffusion coefficient, with many advocating for the amplification of these coefficients when distant from the global optimum. For instance, \cite{Mehl} and \cite{Roberts} employed tempered langevin dynamics, and \cite{Poliannikov} proposed to adjust the diffusion coefficient to be proportional to the inverse of $\|\nabla F(X_t)\|$.
While our focus is on reversible diffusion, the use of non-reversible diffusion for optimization has been advocated in works such as \cite{Ueda,Chii,Nier}. Another related line of research involves controlling temperature. For example, \cite{Gao} introduced a control method from from the perspective of optimal control, and \cite{Ye} proposed continuous tempering for neural networks.

Moreover, Riemannian Langevin Dynamics has been proposed in \cite{Abbati,Xiao,Kelvin}. Their motivation is to utilize an appropriate Riemannian geometry to enhance the convergence of Langevin Dynamics, which subsequently results in a non-constant diffusion coefficient in its implementations. Convergence analyses for these algorithms are also detailed in \cite{Gatmiry,Xiao,Kelvin}, and their approaches have provided valuable insights for our work.

Finally, we introduce two methodologies that closely align with our approach.
Our method draws inspiration from landscape modifications, a concept first introduced in \cite{Fang}. Landscape modifications employ a modified diffusion:
\[
dX_t = -\nabla F(X_t)dt + \sqrt{2\beta^{-1}(f((F(X_t) - c)^+)) + \varepsilon)}\,dW_t,\quad\varepsilon>0
\]
where $f$ is a bounded and twice continuously differentiable function that satisfies $f(0)=f'(0)=f''(0) = 0$. Its stationary distribution is given by
\[
\hat{\nu}(x) \propto \frac{1}{f((F(x) - c)^+)) + 1}
\exp\left\{-\int_{F(x_*)}^{F(x)}\frac{2\beta}{f((F(\xi) - c)^+)) + 1}\,d\xi\right\}
\]
and although this deviates from $\nu$, it retains the same concentration property\cite{Fang}. Given that the diffusion coefficient is larger when far from the optimum, this diffusion modifies the objective function to exhibit a lower energy barrier \cite{choi2,Choi,Fang}. Thanks to this modification, the convergence is accelerated and thus leads to an effective optimization algorithm under low temperature. 

Another closely related approach is outlined in \cite{Engquist}, which employs a diffusion without a drift coefficient:
\[
dX_t = \sqrt{2((F(X_t) - F(x_*))^+))^\gamma + \varepsilon}\,dW_t,\quad \gamma \ge \frac{n}{2},\quad
\varepsilon>0.
\]
This method utilizes a diffusion coefficient similar to that in our proposed method, but surprisingly, it doesn't require information about the derivative $\nabla F(X_t)$. The effectiveness of this functional form for global non-convex optimization supports our choice of diffusion coefficient.

\section{Preliminaries}
Before presenting our main results, we will introduce some preliminary theories and provide definitions for the necessary terms.

\subsection{Reversible Diffusion}

We consider an $n$-dimensional diffusion process $(X_t)_{t\ge 0}$ defined by the stochastic differential equation
\[
dX_t = b(X_t)dt + \sqrt{2\beta^{-1}}\sigma(X_t)dW_t,\quad \beta>0    
\]
where $b:\mathbb{R}^n\to\mathbb{R}^n, \sigma:\mathbb{R}^n\to \mathbb{R}^{n\times n}$ and $(W_t)_{t\ge 0}$ denotes an $n$-dimensional Brownian motion. The generator $\mathcal{L}$ of the diffusion is defined as
\[
\mathcal{L}g(x) := \lim_{t\to 0}\frac{(P_tg)(x)-g(x)}{t} = b(x)\cdot\nabla g(x) + \frac{1}{\beta}\mathrm{Tr}(a(x)\nabla^2 g(x))
\]
where $(P_tg)(x) = \mathbb{E}(g(X_t)|X_0=x), a(x) = \sigma(x)^{\top}\sigma(x)$. The adjoint operator of $\mathcal{L}$, also known as the Fokker-Planck operator, is defined as
\[
\mathcal{L}^*g(x) := \nabla\cdot\left(-b(x)\cdot g(x) + \frac{1}{\beta}\nabla\cdot(a(x)g(x))\right).
\]
We assume that the transition density function $\rho(t, x)$ satisfies the forward Kolmogorov equation:
\[
\frac{\partial}{\partial t}\rho(t, x) = \mathcal{L^*}\rho(t, x)
\]
and that the stationary density function $\nu(x)$ satisfies the stationary Fokker-Planck equation:
\[
\mathcal{L^*}\nu(x) = 0
\]
which is a second order linear partial differential equation.

The class of diffusion processes, which forms the main focus of this paper, is defined as follows

\begin{definition}[Reversible Diffusion]
A stationary diffusion process $(X_t)_{t\ge 0}$ with stationary distribution $\nu$ is said to be reversible if it satisfies the detailed balance condition:
\[
b(x)\nu(x) - \frac{1}{\beta}\nabla\cdot(a(x)\nu(x)) = 0,\quad x\in\mathbb{R}^n.
\]
\end{definition}

\begin{remark}
While it is not intuitive to call this as "reversible", this condition is equivalent to the condition that 
$(X_t)_{t\ge 0}$ is stationary and for every $T > 0$, $X_t$ and $X_{T-t}$ have 
the same distribution\cite{Qian}. 
\end{remark}

Although solving the stationary Fokker-Planck equation for given $b$ and $\sigma$ is challenging, we can readily find a reversible diffusion for an arbitrary stationary distribution \citep{Erdogdu, Ma, Pavliotis}. Specifically, if we let
\[
b(x) = -a(x)\nabla F(x) + \frac{1}{\beta}\nabla\cdot a(x) 
\]
then $(X_t)_{t\ge 0}$ is a reversible diffusion with a stationary distribution of $\nu(x) \propto \exp(-\beta F(x))$. The case of $a(x) = I_n$ corresponds to Langevin Dynamics. There is a degree of freedom in the selection of $a(x)$ and this flexibility enables us to find a diffusion coefficient that is more suited for non-convex optimization than a constant one.

\subsection{Log-Sobolev Inequality}

Let $\rho$ and $\mu$ be probability distributions on $\mathbb{R}^n$ with density functions denoted as $\rho(x)$ and $\nu(x)$, respectively. The KL divergence (Kullback-Leibler divergence, relative entropy) of $\rho$ with respect to $\nu$ is defined as
\[
H_{\nu}(\rho) = \int_{\mathbb{R}^n}\rho(x)\log{\frac{\rho(x)}{\nu(x)}}\,dx.
\]
We will use the KL divergence to measure the distance between probability distributions, and discuss their convergence.
The relative Fisher information of $\rho$ with respect to $\nu$ is defined as
\[
J_{\nu}(\rho) = \int_{\mathbb{R}^n}\rho(x)\left\|\nabla\log{\frac{\rho(x)}{\nu(x)}}\right\|^2dx.
\]

If we denote the probability distribution of $X_t$ as $\rho_t$ and assume the setting of Langevin Dynamics ($a(x) = I_n$), the entropy decreases along the flow:
\[
\frac{d}{dt}H_{\nu}(\rho_t) = -\frac{1}{\beta}J_{\nu}(\rho_t)
\]
and this equation is called the de Bruijin's identity\cite{Bakry}.

\begin{definition}[Log-Sobolev Inequality, LSI]
A distribution $\nu$ on $\mathbb{R}^n$ is said to satisfy the log-Sobolev inequality with a constant $\alpha > 0$ if for all distributions on $\mathbb{R}^n$:
\[
H_{\nu}(\rho)\le \frac{1}{2\alpha} J_{\nu}(\rho).
\]
\end{definition}

It is worth noting that the LSI is applicable for a Gibbs distribution with non-convex potentials, and the inequality is preserved under bounded perturbation\cite{Bakry,Vempala}. Under the LSI, the de Bruijin's identity leads to an exponential convergence of KL divergence:
\[
H_{\nu}(\rho_t) \le e^{-2\alpha\beta^{-1} t}H_{\nu}(\rho_0).
\]
This is the fundamental idea of \cite{Vempala} and our main analysis.

Finally, we define a smoothness condition as follows. 
\begin{definition}[$L$-smoothness]
We say a function $F: \mathbb{R}^n\to \mathbb{R}$ is $L$-smooth if there exists a constant $L > 0$ which satisfies $-LI_n \preceq \nabla^2 F(x) \preceq LI_n$ for all $x\in \mathbb{R}^n$.
\end{definition}

\section{Main Results}
\label{Main Results}
In this section, we introduce our main theory and proposed algorithm.
We propose a reversible diffusion process with
\[
a(x) = (f_{\lambda, \vartheta}((F(x) - c)^+)) + 1)I_n ,\quad f_{\lambda, \vartheta}(x) = \lambda(1 - e^{-\vartheta x^2}) 
\]
where $\displaystyle c > F(x_*), \lambda > 0, \vartheta > 0$. $f_{\lambda, \vartheta}$ is an activation function, and the case of either $\lambda = 0$ or $\vartheta = 0$ corresponds to the langevin dynamics. 
As discussed in the previous section, the drift coefficient is determined by the equation $b(x) = -a(x)\nabla F(x) + \beta^{-1}\nabla\cdot a(x)$.

In this method, we rely on some prior knowledge of $F(x_*)$, and the diffusion coefficient increases when it is far from the global optimum.  While this may seem restrictive, it is worth noting that many practical optimization problems have such prior knowledge (e.g., least squares). Furthermore, even in cases where prior knowledge is unavailable, we can update the parameter $c$ adaptively during the optimization process, making our method applicable to a wider range of problems. We will introduce the adaptive algorithm in \hyperref[3.2]{3.2}.

\subsection{Theory}

In order to implement a discrete-time algorithm, we employ the Euler-Maruyama method:
\[
\hat{X}_{k+1} = \hat{X}_k -\eta(a(\hat{X}_k)\nabla F(\hat{X}_k) - \beta^{-1}\nabla\cdot a(\hat{X}_k))
+ \sqrt{2\beta^{-1}\eta a(\hat{X}_k)}Z_k,\quad k=0,1,\dots
\]
where $Z_k\sim\mathcal{N}(0, I_n)$ is a Gaussian random variable and it is independent from $\hat{X}_k$, and let $\rho_k$ be the probability distribution of $\hat{X}_{k}$.

For the analysis of the Euler-Maruyama method, we consider a stochastic differential equation
\[
dX_t^k = -(a(X_0^k)\nabla F(X_0^k) - \beta^{-1}\nabla\cdot a(X_0^k))dt + \sqrt{2\beta^{-1}a(X_0^k)}dW_t,
\ X_0^k\sim \rho_k,\ t\in[0, \eta]
\]
and let $\rho_t^k$ denote the probability distribution of $X_{t}$. By definition, 
$\rho_0^k = \rho_k, \rho_\eta^k=\rho_{k+1}$ and $X_0^k = \hat{X}_k$. The solution of this stochastic differential equation aligns to the above Euler-Maruyama method, and this leads to $X_\eta^k = \hat{X}_{k+1}$.

For $k = 0,1,\dots$, we define
\begin{eqnarray*}
\gamma_k &=&
\sup\left\{\gamma\in\mathbb{R};
\forall t\in[0,\eta], \gamma J_{\nu}(\rho_t^k)\le
\mathbb{E}\left[a(X_0^k)\left\|\nabla\log{\frac{\rho_t(X_{t}^k)}{\nu(X_{t}^k)}}\right\|^2\right]
\right\} \\
\delta_k &=& \mathbb{E}\left[a(X_0^k)^2\right].
\end{eqnarray*}
Given that $f_{\lambda,\vartheta}(x) \ge 0$, it follows by definition that $\delta_k \ge 1$ and $\gamma_k \ge 1$.
By using $\gamma_k$ and $\delta_k$, we obtain the following theorem.

\begin{theorem}
\label{thm}
Suppose $\nu$ satisfies LSI with constant $\alpha > 0$ and $F$ is $L$-smooth. If $k\ge 2$ and 
\[
0< \eta \le \frac{\delta_k}{(\lambda + 2)^2L}
\vee \frac{\alpha\sqrt{\gamma_k}}{4\beta(\lambda+2)^{3/2} L^2},\quad \vartheta\le\beta^2
\]
then an inequality  
\begin{equation}
H_{\nu}(\rho_{k+1})\le e^{-\frac{\alpha\gamma_k\eta}{\beta}}H_{\nu}(\rho_k)
+ 8\delta_k\eta^2 nL^2.
\end{equation}
holds.
\end{theorem}

\begin{proof}
The proof of this theorem is postponed to the supplementary material.
\end{proof}

\begin{remark}
Since we consider large $\beta$, the fundamental condition for $\eta$ will be $\displaystyle \eta
\le \frac{\alpha\sqrt{\gamma_k}}{4\beta(\lambda+2)^{3/2} L^2}$. In addition, $\lambda = 0$ reverts to the result of \cite{Vempala}. When $\vartheta=\infty$, $a(x)$ becomes $\lambda I_n$, and $\gamma_k = \lambda, \delta_k = \lambda^2$. In this case, it is equivalent to scaling $\eta$ by a factor of $\lambda$.
\end{remark}

In this work, an explicit evaluation of $\gamma_k$ and $\delta_k$ is not provided, and addressing this remains a challenge for future work. In this theorem, we assumed $\eta$ to be bounded by $\lambda^{-3/2}$, but our numerical experiments in Section \hyperref[Numerical Experiments]{4} suggest that this assumption is overly restrictive. This arises from the use of the uniform bound of $f_{\lambda,\vartheta}$ in the theorem's construction, leading to an evaluation that is not as precise as it could be. This issue is closely tied to the evaluation of $\gamma_k$ and $\delta_k$, and therefore, addressing it will be a central focus of future work.

By iterating these results for every $k$, we obtain the following corollary.

\begin{corollary}
Suppose $\nu$ satisfies LSI with constant $\alpha > 0$ and $F$ is $L$-smooth. If $k\ge 2$ and 
\[
0< \eta \le \frac{1}{(\lambda + 2)^2L}
\vee \frac{\alpha}{4\beta(\lambda+2)^{3/2} L^2},\quad \vartheta\le\beta^2
\]
then an inequality
\[
H_{\nu}(\rho_k)\le e^{-\frac{\alpha\eta c_k}{\beta}}H_{\nu}(\rho_0)
+ 8\eta^2 n L^2d_k
\]
holds, where
\[
c_k = \sum_{j=0}^{k-1}\gamma_j,\quad d_k = \sum_{j=0}^{k-2}e^{-\alpha\beta^{-1}\eta(\gamma_{j+1}+\cdots+\gamma_{k-1})}\delta_j
+\delta_{k-1}.
\]
\end{corollary}

When $\rho_k$ is far from $\nu$, both $\gamma_k$ and $\delta_k$ tend to become large. The parameter 
$c_k$ contributes to faster convergence, while $d_k$ exacerbates the discretization error.
However, these effects are not equivalent. This is because the term $\displaystyle e^{-\alpha\beta^{-1}\eta(\gamma_{k-1}+\cdots+\gamma_{j+1})}$ multiplies each $\delta_j$ and effectively suppress the contribution of early iterations to the discretization error. Moreover, when $\rho_k$ is sufficiently near $\nu$, our proposed method aligns with the corresponding Langevin Dynamics. Consequently, for sufficiently large $k$, we can deduce that $\gamma_k = \delta_k = 1$. This implies that the discretization error does not significantly increase compared to the corresponding Langevin Dynamics if we run large iterations.

When $\beta$ is sufficiently large, the $\beta^{-1}\nabla\cdot a(\hat{X}_k)$ term in the drift coefficient becomes negligible compared to the $a(\hat{X}_k)\nabla F(\hat{X}_k)$ term. This can effectively be interpreted as an adaptive step size version of the corresponding Langevin Dynamics:
\[
\eta \to \eta (f_{\lambda, \vartheta}((F(\hat{X}_k) - c)^+)) + 1).
\]
Therefore, our proposed method is actually changing the time scale of the diffusion. When it is far from the global optimum, the step size becomes large, effectively "accelerating the clock speed". When it is near the global optimum, the step size becomes the same as the corresponding Langevin Dynamics, thereby "recovering the regular clock speed". Thanks to the time change, the diffusion can travel for a long time. This is fundamentally important, as the core idea behind Langevin Dynamics based algorithms is to sample from their stationary distribution. From this perspective, the $\beta^{-1}\nabla\cdot a(\hat{X}_k)$ term can be viewed as a correction term for the implementation of the adaptive step.

Our strategy is to set the parameter $\lambda$ proportional to the inverse temperature $\beta$. For instance, if we assign $\lambda = \beta$, the value of $f_{\lambda,\vartheta}$ closely approximates $\beta$, and the diffusion essentially mimics the environment of Langevin Dynamics with $\beta = 1$ when distant from the optimum.

\subsection{Proposed Algorithm}
\label{3.2}
In order to introduce the adaptive algorithm, we define $h:\mathbb{R}^d\times\mathbb{R}\to \mathbb{R}$ as 
\[
h(x, y) = (f_{\lambda, \vartheta}((F(x) - y)^+)) + 1)
\]
and the proposed algorithm is shown below. We execute $M$ samples concurrently, estimating $c$ using the minimum of the current values of the objective function in each iteration. Similar adaptive estimation of the threshold was employed in \cite{Engquist}.

It is important to note that if the initial distribution is a delta distribution on a single local optimum, the diffusion is likely to be trapped in the local optimum. This is because the estimation of $c$ would be equivalent to the value of the local optimum, leading to a small diffusion coefficient which is equivalent to the corresponding Langevin Dynamics.

\begin{algorithm}[H]
\label{Algo}
\caption{Adaptive Volatility Optimization (AdaVol)}
\begin{algorithmic}[1]
    \Require initial distribution $\rho_0$, inverse temperature $\beta$, 
    step size $\eta$, parameters $\lambda , \vartheta$, sample size $M$
    \For{$i\in \{1,\dots, M\}$}
    \State randomly draw $x_0^i\sim \rho_0$
    \EndFor
    \State $c_0 = \min\{F(x_0^1),\dots,F(x_0^M)\}$
    \For{$j\in \{1,\dots, N\}$}
    \For{$i\in \{1,\dots, M\}$}
    \State randomly draw $z_j^{i}\sim \mathcal{N}(0, I_n)$
    \State $\displaystyle x_{j+1}^i = x_{j}^i +\eta\left(h(x_j^i, c_j)\nabla F(x_j^i) - \beta^{-1}\frac{\partial h}{\partial x}(x_j^i, c_j)\right) 
    + \sqrt{2\beta^{-1}\eta h(x_{j}^i, c_j)}z_j^{i}$
    \EndFor
    \State $c_{j+1} = \min\{c_{j}, F(x_{j+1}^1),\dots,F(x_{j+1}^M)\}$
    \EndFor
    \Ensure $\displaystyle\hat{x} = \frac{1}{M}\sum_{i=1}^{M}x_{N}^i$
\end{algorithmic}
\end{algorithm}

Although the gradient of $h$ needs to be calculated at each iterations, it is given by
\[
\frac{\partial h}{\partial x}(x_j^i, c_j)
= -2\lambda\vartheta(F(x_j^i) - c_j)^+ e^{-\vartheta((F(x_j^i) - c_j)^+)^2}\nabla F(x_j^i)
\]
and the only necessary gradient is $\nabla F(x_j^i)$. Consequently, AdaVol does not require additional gradient calculation costs compared to the Langevin Dynamics.

In addition, as we relies on the empirical mean for the estimation of $x_*$, this can lead to a biased estimation when $\beta$ is not sufficiently large and the objective function is asymmetric. Utilizing a mode estimator could be a more effective alternative, but this consideration lies beyond the scope of the current work.

As a final remark, a discrepancy exists between the setting of Theorem \hyperref[thm]{3.1} and that of AdaVol, given that the coefficients in AdaVol are path-dependent. 
However, our numerical experiments in Section \hyperref[Numerical Experiments]{4} suggest that AdaVol exhibits almost the same convergence property as that of fixed $c$ versions.

\section{Numerical Experiments}
\label{Numerical Experiments}

In this section, we illustrate the effectiveness of our proposed method (AdaVol) through numerical experiments. We employ the Rastrigin function $F:\mathbb{R}^n\to\mathbb{R}$;
\[
F(x) = 5n + \sum_{k=1}^{n}(x_k - 2)^2 - 5\sum_{k=1}^{n}\cos(2\pi(x_k - 2))  
\]
as a test function. This function is highly multimodal, and standard gradient descent algorithms often fail to reach its global optimum, and it is also used in \cite{Engquist2,Engquist}. Also, this function satisfies the assumption of Theorem \hyperref[thm]{1}, and it has a unique global optimum of $0$ for 
$x_* = (2,\dots,2)^\top\in \mathbb{R}^n$. 

For AdaVol, we set the parameters as follows: step size $\eta = 10^{-5}$, inverse temperature $\beta = 10^4$, sample size $M = 10^2$ and $\lambda = 10^4,\vartheta = 1$ for the activation function. 
The initial distribution is set to be $\mathcal{N}(10^3, 10I_2)$. Note that we set the parameter $\lambda$ equal to $\beta$. This choice of $\lambda$ ensures that the diffusion coefficient becomes as large as that of $\beta = 1$ when far from the global optimum.

In Figure \hyperref[1]{1}, we contrast AdaVol with a Langevin Dynamics based algorithm ($\vartheta = 0$ and $\beta = 1, 10^4$). The figure demonstrates that AdaVol matches the speed of Langevin Dynamics with $\beta = 1$ while also attaining a lower objective function value. Conversely, Langevin Dynamics exhibits slow convergence when $\beta = 10^4$.

\begin{figure}[H]
\label{1}
  \centering
  \includegraphics[width=5.25cm]{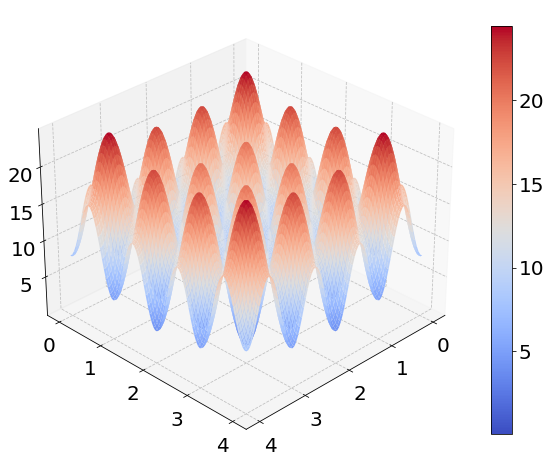}
  \includegraphics[width=8.5cm]{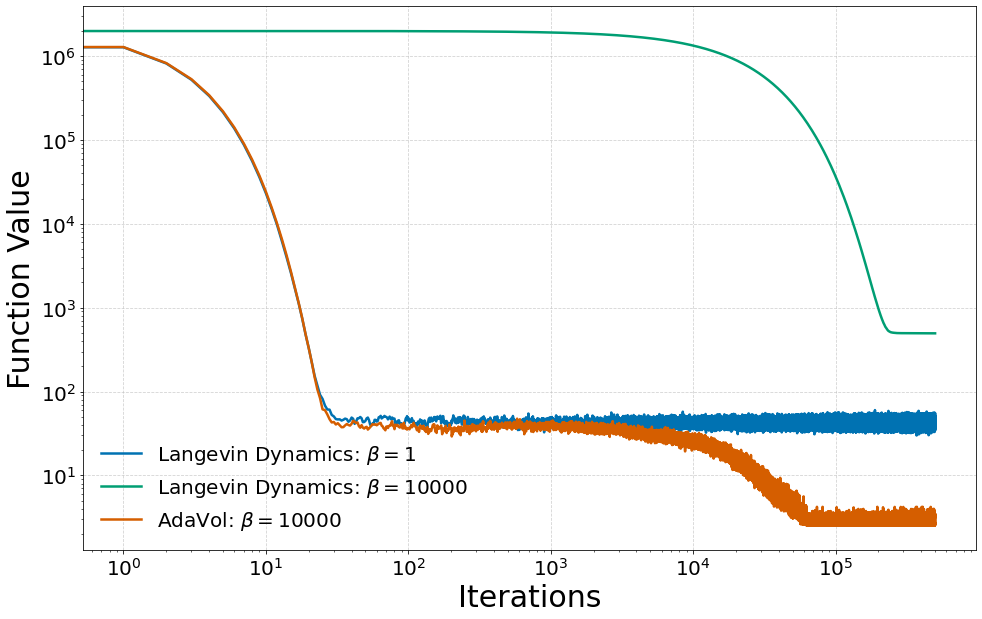}
  \caption{The left plot shows the landscape of Rastrigin function. The right plot compares the proposed method (AdaVol) with existing methods (Langevin Dynamics).}
\end{figure}

Next, we investigate the impact of prior knowledge concerning $F(x_*)$. Figure \hyperref[2]{2} shows that with such a prior knowledge, we can achieve improved convergence. Additionally, it's important to note that while setting a larger $c$ results in quicker convergence, this choice makes it more challenging to distinguish between local optima near the global optimum when the objective value falls below $c$.

\begin{figure}[H]
\label{2}
  \centering
  \includegraphics[width=8.5cm]{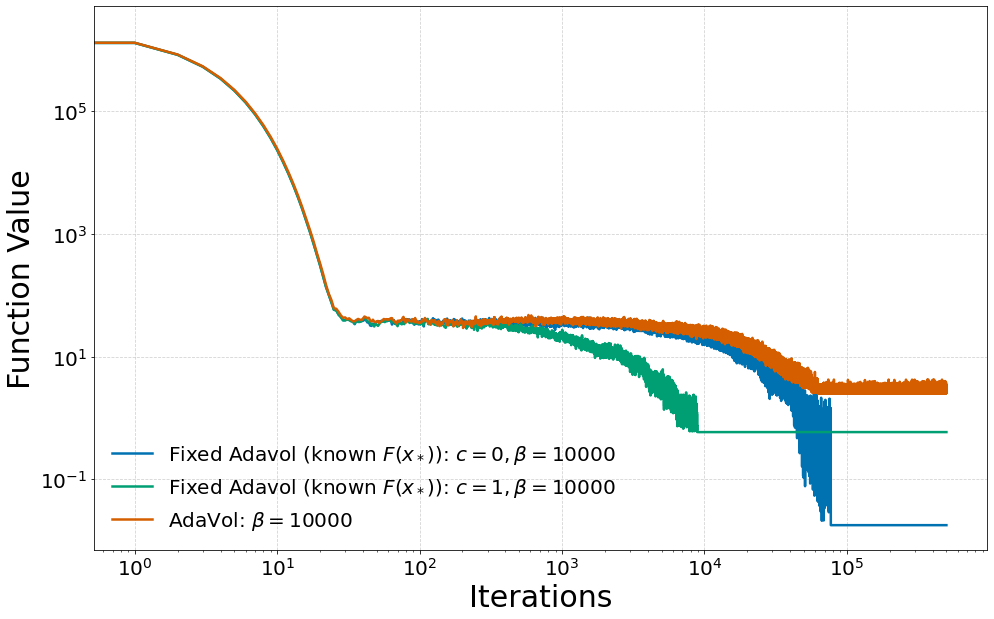}
  \caption{The plot compares fixed and adaptive $c$ scenarios.}
\end{figure}

Our numerical experiments demonstrate the superiority of AdaVol over traditional methods, suggesting its wide applicability in non-convex optimization. However, while the choice of the parameter $\lambda$ is guided by the inverse temperature, the selection of the parameter $\vartheta$ remains unclear. In the above setting ($\vartheta = 1$), the value of $f_{\lambda,\vartheta}$ is almost equal to $\lambda$ when distant from the global optimum, and the choice of $\vartheta$ does not significantly impact the performance.
This is due to the exponential activation of $f_{\lambda, \vartheta}$, which leads to a nearly instantaneous shift in $a$.
For a objective function which has small value and is potentially sensitive to $\vartheta$, the choice of $\vartheta$ should be discussed separately, but this is beyond the scope of our work.

\section{Conclusion and Future Work}

In this paper, we have introduced a noble algorithm for global non-convex optimization (AdaVol), and provided its convergence analysis. Our numerical experiments shows that the AdaVol outperforms traditional Langevin Dynamics based algorithms, and the results is consistent with our theory. 

However, the theoretical analysis is still incomplete. Specifically, revealing the interplay between $\gamma_k$ and $\delta_k$ presents a significant challenge and is a subject for future work. In addition, our choice of the diffusion coefficient is one example, and therefore, the pursuit of the optimal selection and its criteria constitutes the ultimate goal of our framework

For practical applications in machine learning problems, we need to update our methods to the setting where we do not have a full access to the gradient information of the objective function, since the calculation of exact gradient is computationally too expensive\cite{Taiji,Welling}. Applying our method to the loss function of neural networks also presents an avenue for future work.

{
\small
\bibliography{bibliography}
}

\appendix

\section{Preparation for the Proof}

\begin{lemma}
\label{A.1}
Suppose $\nu$ satisfies LSI with constant $\alpha > 0$ and $F$ is $L$-smooth. For $X\sim \nu$,
\[
\mathbb{E}[\|\nabla F(X)\|^2] \le \frac{Ln}{\beta}
\]
holds.
\end{lemma}

\begin{proof}
The proof of this lemma is found in \cite{Taiji}. By integration by parts, we obtain
\[
\mathbb{E}[\|\beta(\nabla F(X))\|^2] = \mathbb{E}[\|\Delta(\beta F(X))\|].
\]
Since $\nabla^2F(x)\preceq LI_n$, we have $\Delta F(x)\le nL$. Therefore,
\[
\mathbb{E}[\|\nabla F(X)\|^2] = \frac{1}{\beta}\mathbb{E}[\|\Delta F(X)\|] \le \frac{Ln}{\beta}
\]
holds.
\end{proof}

\begin{lemma}
\label{A.2}
Suppose $\nu$ satisfies LSI with constant $\alpha > 0$ and $F$ is $L$-smooth. For any distribution $\rho$ on $\mathbb{R}^n$,
\[
\mathbb{E}[a(X)^2\|\nabla F(X)\|^2]\le 2(\lambda + 1)^2L
\left(\frac{2L}{\alpha}H_{\nu}(\rho) + \frac{n}{\beta}\right),\quad
X\sim \rho
\]
holds.
\end{lemma}

\begin{proof}
The proof of this lemma is based on \cite{Vempala}. Let $X_* \sim \nu$ and assume that $(X, X_*)$ is an optimal coupling which satisfies
\[
\mathbb{E}[\|X - X_*\|^2] = W_2(\rho, \nu)^2
\]
where $W_2$ is a Wasserstein distance between $\rho$ and $\nu$. Given that $\nabla F$ is L-Lipschitz, we obtain
\[
\|\nabla F(X)\|\le L\|X - X_*\| + \|\nabla F(X_*)\|.
\]
Therefore, by applying Talagrand's inequality and Lemma \hyperref[A.1]{A.1}, 
\begin{eqnarray*}
\mathbb{E}[a(X)^2\|\nabla F(X)\|^2] &\le& 2(\lambda + 1)^2 L^2 W_2(\rho, \nu)^2 + 2(\lambda + 1)\mathbb{E}[\|\nabla F(X_*)\|^2]\\
&\le& \frac{4(\lambda + 1)^2 L^2}{\alpha}H_{\nu}(\rho) + \frac{2(\lambda + 1)^2Ln}{\beta}
\end{eqnarray*}
holds.
\end{proof}

\begin{lemma}
\label{A.3}
if $k\ge 2$ and 
\[
t \le \frac{\delta_k}{(\lambda + 2)^2L},\quad \theta\le\beta^2
\]
then an inequality
\begin{eqnarray*}
\mathbb{E}\left[\left\langle
a(X_0^k)(\nabla F(X_t^k) - \nabla F(X_0^k)), \nabla\log{\frac{\rho_t(X_t^k)}{\nu(X_t^k)}}
\right\rangle\right]
&\le& \frac{4\beta(\lambda + 2)^3L^4t^2}{\alpha}H_{\nu}(\rho_k) + 4\delta_k L^2nt \\
&\,& + \frac{1}{4\beta}\mathbb{E}\left[
a(X_0^k)\left\|\nabla\log{\frac{\rho_t(X_t^k)}{\nu(X_t^k)}}\right\|^2\right]
\end{eqnarray*}
holds.
\end{lemma}

\begin{proof}
The proof of this lemma is based on \cite{Vempala}. 
Since $\displaystyle xy \le \beta x^2 + \frac{1}{4\beta}y^2$ and $\nabla F$ is L-Lipschitz we obtain
\begin{eqnarray*}
\mathbb{E}\left[\left\langle
a(X_0^k)(\nabla F(X_t^k) - \nabla F(X_0^k)), \nabla\log{\frac{\rho_t(X_t^k)}{\nu(X_t^k)}}
\right\rangle\right]
&\le& \beta\mathbb{E}[a(X_0^k)\|\nabla F(X_t^k) - \nabla F(X_0^k)\|^2]\\
&\,& + \frac{1}{4\beta}\mathbb{E}\left[
a(X_0^k)\left\|\nabla\log{\frac{\rho_t(X_t^k)}{\nu(X_t^k)}}\right\|^2\right]\\
&\le& \beta(\lambda + 1)L^2\mathbb{E}[\|X_t^k - X_0^k\|^2]\\
&\,& + \frac{1}{4\beta}\mathbb{E}\left[
a(X_0^k)\left\|\nabla\log{\frac{\rho_t(X_t^k)}{\nu(X_t^k)}}\right\|^2\right].
\end{eqnarray*}
We have the expression
\[
X_t^k = X_0^k - t(a(X_0^k)\nabla F(X_0^k) - \beta^{-1}\nabla\cdot a(X_0^k))
+ \sqrt{2\beta^{-1}ta(X_0^k)}Z_0,\quad Z_0\sim \mathcal{N}(0, I_n).
\]
Given that 
\begin{eqnarray*}
\mathbb{E}[\|\nabla\cdot a(X_0^k)\|^2] \le \theta\mathbb{E}[\|\nabla F(X_0^k)\|^2],\quad
t \le \frac{\delta_k}{(\lambda + 2)^2L},\quad\theta\le\beta^2
\end{eqnarray*}
and applying Lemma \hyperlink{A.2}{A.2},
\begin{eqnarray*}
\mathbb{E}[a(X_0^k)\|X_t^k - X_0^k\|^2] &=& \mathbb{E}\left[\left\|ta(X_0^k)(a(X_0^k)\nabla F(X_0^k) - \beta^{-1}\nabla\cdot a(X_0^k))
+ \sqrt{2\beta^{-1}a(X_0^k)}Z_0\right\|^2\right]\\
&\le& t^2\mathbb{E}[a(X_0^k)^3\|\nabla F(X_0^k)\|^2] + \frac{t^2}{\beta^2}\mathbb{E}[a(X_0^k)\|\nabla\cdot a(X_0^k)\|^2]
+ \frac{2\delta_kLnt}{\beta}\\
&\le& t^2\left((\lambda + 1)^2 + \frac{\theta}{\beta^2}\right)
\left(\frac{4L^2}{\alpha}H_{\nu}(\rho_k) + \frac{2n}{\beta}\right)
+ \frac{2\delta_kLnt}{\beta}\\
&\le& \frac{4(\lambda + 2)^2L^2t^2}{\alpha}H_{\nu}(\rho_k) + \frac{4\delta_kLnt}{\beta}.
\end{eqnarray*}
Therefore, we obtain
\begin{eqnarray*}
\mathbb{E}\left[\left\langle
a(X_0^k)(\nabla F(X_t^k) - \nabla F(X_0^k)), \nabla\log{\frac{\rho_t(X_t^k)}{\nu(X_t^k)}}
\right\rangle\right]
&\le& \frac{4\beta(\lambda + 2)^3L^4t^2}{\alpha}H_{\nu}(\rho_k) + 4\delta_k L^2nt \\
&\,& + \frac{1}{4\beta}\mathbb{E}\left[
a(X_0^k)\left\|\nabla\log{\frac{\rho_t(X_t^k)}{\nu(X_t^k)}}\right\|^2\right]
\end{eqnarray*}
\end{proof}

\section{Main Proof}

\begin{theorem}
Suppose $\nu$ satisfies LSI with constant $\alpha > 0$ and $F$ is $L$-smooth. If $k\ge 2$ and 
\[
0< \eta \le \frac{\delta_k}{(\lambda + 2)^2L}
\vee \frac{\alpha\sqrt{\gamma_k}}{4\beta(\lambda+2)^{3/2} L^2},\quad \theta\le\beta^2
\]
then an inequality  
\begin{equation}
H_{\nu}(\rho_{k+1})\le e^{-\frac{\alpha\gamma_k\eta}{\beta}}H_{\nu}(\rho_k)
+ 8\delta_k\eta^2 nL^2.
\end{equation}
holds.
\end{theorem}

\begin{proof}
The framework of this proof is taken from \cite{Vempala}.
Let $\rho_{0t}^k(x_0,x_t)$ be the joint density of $(X_0^k, X_t^k)$. The Fokker-Planck equation for the conditional density $\rho_{t\vert 0}^k(x_t\vert x_0)$ is
\begin{equation}
\frac{\partial}{\partial t}\rho_{t\vert 0}^k(x_t\vert x_0) = 
-\nabla\cdot(a(x_0)\nabla F(x_0)\rho_{t\vert 0}^k(x_t\vert x_0))
+ \beta^{-1}\nabla\cdot(a(x_0)\nabla\rho_{t\vert 0}^k(x_t\vert x_0)).
\end{equation}
From this equation, we obtain the Fokker-Planck equation:
\begin{eqnarray}
\frac{\partial}{\partial t}\rho_{t}^k(x) &=& \int_{\mathbb{R}^n}\frac{\partial \rho_{t\vert 0}^k(x_t\vert x_0)}{\partial t}\rho_0^k(x_0)\,dx_0\nonumber\\ 
&=& \nabla\cdot\left(\mathbb{E}[a(x_0)\nabla F(x_0)\vert X_t=x]\rho_t^k(x)\right)
+ \beta^{-1}\nabla\cdot\left(\mathbb{E}[a(X_0)\vert X_t=x]\nabla\rho_t^k(x)\right).
\end{eqnarray}
Therefore, the time derivative of KL divergence along the diffusion is given by:
\begin{eqnarray}
\frac{d}{dt}H_{\nu}(\rho_t^k) &=& \int_{\mathbb{R}^n}\frac{\partial\rho_t^k(x)}{\partial t}\log{\frac{\rho_t^k(x)}{\nu(x)}}\,dx\nonumber \\
&=& \int_{\mathbb{R}^n}\left(\nabla\cdot\left(\rho_t^k(x)A(x_0,x)
\right)\right)\log{\frac{\rho_t^k(x)}{\nu(x)}}\,dx\nonumber\\
&=& -\int_{\mathbb{R}^n}\rho_t^k(x)\left\langle
A(x), \nabla\log{\frac{\rho_t^k(x)}{\nu(x)}}
\right\rangle\,dx
\end{eqnarray}
where
\begin{eqnarray}
A(x) &=& \mathbb{E}[a(X_0^k)\vert X_t^k=x] + \beta^{-1}\mathbb{E}[a(X_0^k)\vert X_t^k=x]\frac{\nabla\rho_t^k(x)}{\rho_t^k(x)}
\nonumber\\
&=& \mathbb{E}[a(X_0^k)(\nabla F(X_0^k) - \nabla F(X_t))\vert X_t^k=x] 
+ \beta^{-1}\mathbb{E}[a(X_0^k)\vert X_t^k=x]\nabla\log{\frac{\rho_t^k(x)}{\nu(x)}}.\nonumber
\end{eqnarray}

From Lemma \hyperref[A.3]{A.3} and the log-Sovolev inequality, we obtain
\begin{eqnarray*}
\frac{d}{dt}H_{\nu}(\rho_t^k) &\le&
-\frac{3}{4\beta}\mathbb{E}\left[
a(X_0^k)\left\|\nabla\log{\frac{\rho_t(X_t^k)}{\nu(X_t^k)}}\right\|^2\right]
+\frac{4\beta(\lambda + 2)^3L^4t^2}{\alpha}H_{\nu}(\rho_k) + 4\delta_k L^2nt\\
&\le& -\frac{3\alpha\gamma_k}{2\beta}H_{\nu}(\rho_t^k)
+\frac{4\beta(\lambda + 2)^3L^4t^2}{\alpha}H_{\nu}(\rho_k) + 4\delta_k L^2nt
\end{eqnarray*}
Consequently,
\begin{eqnarray}
e^{\frac{3\alpha\gamma_k\eta}{2\beta}}H_{\nu}(\rho_{k+1})
&\le& H_{\nu}(\rho_k) + \int_{0}^{\eta}e^{\frac{3\alpha\gamma_k\xi}{2\beta}}
\left(\frac{4\beta(\lambda + 2)^3\eta^2L^4}{\alpha}H_{\nu}(\rho_k) + 4\delta_k\eta L^2n\right)d\xi\nonumber\\
&\le& \left(1 + \frac{8\beta(\lambda + 2)^3\eta^3L^4}{\alpha}\right)H_{\nu}(\rho_k)
+ 8e^{\frac{3\alpha\gamma_k\eta}{2\beta}}\delta_k\eta^2 L^2n.
\end{eqnarray}
Since $\displaystyle \eta\le \frac{\alpha\sqrt{\gamma_k}}{4\beta(\lambda+2)^{3/2} L^2}$, we have
\begin{equation}
H_{\nu}(\rho_{k+1}) \le e^{-\frac{\alpha\gamma_k\eta}{\beta}}H_{\nu}(\rho_{k})
+ 8\delta_k\eta^2 nL^2.
\end{equation}
\end{proof}

%%%%%%%%%%%%%%%%%%%%%%%%%%%%%%%%%%%%%%%%%%%%%%%%%%%%%%%%%%%%
\end{document}